\documentclass[12pt]{amsart}

\usepackage{amsmath,amssymb,amsmath,epic,lscape}

\newtheorem{theorem}{Theorem}[section]
\newtheorem{proposition}[theorem]{Proposition}

\theoremstyle{definition}

\theoremstyle{remark}

\numberwithin{equation}{section}

\def\DJ{{\hbox{D\kern-.8em\raise.15ex\hbox{--}\kern.35em}}}
\def\DJo{$\;$\kern-.4em
    \hbox{D\kern-.8em\raise.15ex\hbox{--}\kern.35em okovi\'c}}

\renewcommand{\subjclassname}{\textup{2000} Mathematics Subject
Classification }

\begin{document}

\title[Hadamard matrices of small order]
{Hadamard matrices of small order and Yang conjecture}

\author[D.\v{Z}. \DJ okovi\'{c}]
{Dragomir \v{Z}. \DJ okovi\'{c}}

\address{Department of Pure Mathematics, University of Waterloo,
Waterloo, Ontario, N2L 3G1, Canada}

\email{djokovic@uwaterloo.ca}

\keywords{Base sequences, normal and near-normal sequences, T-sequences,
orthogonal designs, Williamson-type matrices, Yang conjecture}

\date{}

\begin{abstract}
We show that 138 odd values of $n<10000$ for which a Hadamard 
matrix of order $4n$ exists have been overlooked in the recent 
handbook of combinatorial designs. There are four additional odd 
$n=191,5767,7081,8249$ in that range for which Hadamard matrices 
of order $4n$ exist. 
There is a unique equivalence class of near-normal sequences 
$NN(36)$, and the same is true for $NN(38)$ and $NN(40)$. This 
means that the Yang conjecture on the existence of near-normal 
sequences $NN(n)$ has been verified for all even $n\le40$, but it still remains open.
\end{abstract}

\maketitle
\subjclassname{ 05B20, 05B30 }
\vskip5mm

\section{Introduction}

Recall that a Hadamard matrix of order $m$ is a 
$\{\pm1\}$-matrix $A$ of size $m\times m$ such that $AA^T=mI_m$, 
where $T$ denotes the transpose and $I_m$ the identity matrix.
Let us denote by $HM(m)$ the set of Hadamard matrices of order 
$m$. By abuse of language, we say that $HM(m)$ exist if 
$HM(m)\ne\emptyset$. If $m>2$ and $HM(m)$ exist, then $m$ is 
divisible by 4. In the recent handbook \cite[pp. 278--279]{HCD} 
one finds a table of all odd integers $n<10000$ and exponents 
$t$ $(2\le t\le 8)$ for which it is known that $HM(2^t n)$ 
exist. The Hadamard conjecture asserts that we should always 
have $t=2$ (for $n>1$). We shall refer to the cases with $t\ge3$ 
as bad cases. The good cases, i.e., those with $t=2$ are 
indicated in the table by a dot.
The table has 5000 entries of which 1006 are bad. In this
note we point out that 138 of these 1006 cases are in fact good.
Additional 4 cases can be also eliminated,
reducing the number of bad cases to 864.

In the next section we recall a well known construction 
(see Proposition \ref{fakt}) of $HM(4n)$ which uses Yang 
multiplication, T-sequences, orthogonal designs, and 
Williamson-type matrices. In section \ref{losi} 
we list 138 bad cases and invoke Proposition \ref{fakt} 
to show that they should have been classified as good cases. 
We also mention four additional bad cases that we can eliminate. 
In section \ref{40-Yang} we show that there exist near-normal
sequences $NN(n)$ for $n=32,34,36,38$ and $40$. 
Thus, despite of our efforts to find a counter-example, the Yang 
conjecture on near-normal sequences still remains open.

\section{Preliminaries}

Let us introduce the following notation for the sets of some
important combinatorial objects (for their definitions see
\cite{HCD,DZ2,DZ3,SY}):
$$
\begin{array}{ll}
GS(g) & \text{Golay sequences of length } g \\
TS(t) & \text{T-sequences of length } t \\
BS(r,s) & \text{Base sequences of lengths } r,s \\
NS(l) & \text{Normal sequences, as a subset of } BS(l+1,l) \\
NN(l) & \text{Near-normal sequences, as a subset of } BS(l+1,l) \\
OD(4d) & \text{Orthogonal designs } OD(4d;d,d,d,d) \\
WT(w) & \text{Williamson-type matrices of order } w \\
BHW(4h) & \text{Baumert--Hall--Welch arrays of order } 4h
\end{array}
$$

It is well known that there exist constructions (i.e., maps)
\begin{eqnarray}
&& GS(g) \to NS(g),  \label{Golay1} \\
&& GS(g) \to BS(g,1), \label{Golay2} \\
&& BS(r,s) \to TS(r+s), \label{bazis} \\
&& NS(l) \times BS(r,s) \to TS((2l+1)(r+s)), \label{Y1} \\
&& NN(l) \times BS(r,s) \to TS((2l+1)(r+s)), \label{Y2} \\
&& BHW(4h) \times TS(t) \to OD(4ht), \label{OD1} \\
&& OD(4d) \times WT(w) \to HM(4dw). \label{OD2} 
\end{eqnarray}

The first three constructions are elementary. They are given by
\begin{eqnarray*}
(A;B) &\to& (A,+;A,-;B;B) \\
(A;B) &\to& (A;B;+;+) \\
(A;B;C;D) &\to& ((A+B)/2,0_s;(A-B)/2,0_s; \\
 && \quad\quad 0_r,(C+D)/2;0_r,(C-D)/2),
\end{eqnarray*}
where $+$ and $-$ stand for $+1$ and $-1$, respectively, comma
denotes the concatenation of sequences, and the symbol $0_d$ denotes
the sequence of $d$ zeros. The arithmetic operations on sequences are
performed component-wise. The constructions (\ref{Y1}) and 
(\ref{Y2}) are due to Yang \cite{CHY}.
For the ``plug in'' constructions (\ref{OD1}) and (\ref{OD2}) 
see \cite{SY} Theorems 3.10 and 3.8, respectively.

Recall that $BHW(4h)$ exist for $h\in\{1,5,9\}$ (see \cite{SY}). 
An integer $g$ is a Golay number if $GS(g)$ exist. It is known that
$2^a 10^b 26^c$, $a,b,c\ge0$ integers, are Golay numbers.
An odd integer $y=2l+1$ is a Yang number if $NS(l)$ or $NN(l)$ exist.
The following well known fact is an immediate consequence.

\begin{proposition} \label{fakt}
If $y$ is a Yang number and $BHW(4h),BS(r,s)$ and $WT(w)$ exist, then
$HM(4n)$ exist for $n=yh(r+s)w$.
\end{proposition}
\begin{proof}
We first apply the construction (\ref{Y1}) or (\ref{Y2}), whichever is
appropriate, to obtain T-sequences of length $t=y(r+s)$. Next, 
we apply the construction (\ref{OD1}) to obtain an $OD(4ht)$. 
Finally, the construction (\ref{OD2}) produces an $HM(4n)$.
\end{proof}

\section{An update of the list of bad cases} \label{losi}

Let $\Delta$ be the following set of 138 odd integers $<10000$:
\begin{eqnarray*}
&& \{1397, 2159, 2413, 2773, 2921, 3175, 3953, 4053, 4083, 4097, \\
&& 4181, 4227, 4307, 4389, 4439, 4453, 4479, 4495, 4499, 4589, \\
&& 4633, 4659, 4747, 4765, 4859, 4921, 4981, 5017, 5165, 5199, \\
&& 5201, 5207, 5211, 5259, 5317, 5359, 5363, 5379, 5383, 5411, \\
&& 5461, 5545, 5567, 5597, 5619, 5667, 5709, 5825, 5875, 5913, \\
&& 5915, 5965, 5969, 5979, 5989, 6001, 6059, 6129, 6341, 6351, \\
&& 6369, 6495, 6523, 6605, 6667, 6693, 6707, 6731, 6743, 6755, \\
&& 6805, 6813, 6893, 6953, 6985, 6989, 6995, 7045, 7093, 7223, \\
&& 7325, 7373, 7387, 7413, 7427, 7439, 7471, 7493, 7505, 7571, \\
&& 7613, 7633, 7709, 7765, 7913, 7953, 8033, 8131, 8155, 8197, \\
&& 8299, 8327, 8465, 8477, 8485, 8503, 8509, 8579, 8589, 8633, \\
&& 8655, 8665, 8743, 8833, 8899, 8917, 9005, 9065, 9071, 9083, \\
&& 9087, 9093, 9107, 9169, 9273, 9325, 9365, 9407, 9445, 9485, \\
&& 9515, 9527, 9549, 9553, 9827, 9881, 9959, 9965 \}.
\end{eqnarray*}

According to \cite[pp 278--279]{HCD} all these cases are bad, i.e.,
they all have $t\ge3$.

We claim that the existence of $HM(4n)$ for $n\in\Delta$ has been known
for some time and that all these cases should have been classified as
good. To prove this claim, it suffices to apply the above proposition.
We used only the facts known for a few years prior to the publication 
of \cite{HCD}. In more detail, we used the existence of 
$BS(l+1,l)$ for $l\le35$,
$BS(2l-1,l)$ for even $l\le36$, 
$BS(g_1,g_2)$ for $g_1$ and $g_2$ Golay numbers, and
$NN(l)$ for even $l\le30$.
For Williamson-type matrices $WT(w)$,
we used the listing for odd $w<2000$ given in \cite[Table A.1]{SY} 
with only two changes. Namely we used the fact that $WT(w)$ exist 
for $w=35$ and $w=127$, proven in \cite{DZ1} and
recorded in \cite[Table V.1.50, p. 277]{HCD} as well.
 
Since the verification is of routine nature and tedious, 
we shall just give a few examples (see Table 1) and list the 
acceptable choices for the parameters $y,h,(r,s),w$. 
In some cases there are several such choices, which may give 
different constructions for $HM(4n)$.  

Since $2048$ and $2600$ are 
Golay numbers, it follows from (\ref{Golay1}) that 
$4097=2 \cdot 2048 +1$ and $5201=2\cdot 2600 +1$ are Yang numbers. 
The $BS(r,s)$ that occur in Table 1 can be found in many places,
e.g., \cite{DZ3,SY}, except for the case $(r,s)=(34,33)$. In the 
exceptional case the base sequences were constructed in \cite{KS},
see also below.
Note that $BS(100,1)$, used for the case $n=7373$, exist because 
of (\ref{Golay2}). 

\begin{center}
\bf{Table 1: Parameters for the construction of $HM(4n)$}
\[
\begin{array}{lllll|lllll}
\\
n & y & h & (r,s) & w & n & y & h & (r,s) & w \\
\hline 
&&&&&&&&& \\
2773 & 59 & 1 & (24,23) & 1 &	4495 & 31 & 5 & (15,14) & 1  \\
3953 & 59 & 1 & (34,33) & 1 &	     & 31 & 5 & (1,0) & 29  \\
4097 & 4097 & 1 & (1,0) & 1 &	     & 31 & 1 & (15,14) & 5  \\
     & 1 & 1 & (2049,2048) & 1 &   & 31 & 1 & (3,2) & 29  \\
4389 & 19 & 1 & (17,16) & 7 &	     & 29 & 5 & (16,15) & 1  \\
     & 19 & 1 & (11,10) & 11 &     & 29 & 5 & (1,0) & 31  \\
     & 19 & 1 & (6,5) & 21 & 	     & 29 & 1 & (3,2) & 31  \\
     & 19 & 1 & (4,3) & 33 & 	     & 29 & 1 & (16,15) & 5  \\
     & 11 & 1 & (29,28) & 7 & 5201 & 5201 & 1 & (1,0) & 1  \\
     & 11 & 1 & (11,10) & 19 &     & 1 & 1 & (2601,2600) & 1 \\
     & 11 & 1 & (10,9) & 21 & 5875 & 25 & 5 & (24,23) & 1  \\
     & 11 & 1 & (4,3) & 57 & 	     & 25 & 1 & (24,23) & 5  \\
     & 11 & 1 & (1,0) & 399 &	     & 5 & 1 & (24,23) & 25  \\
     & 7 & 1 & (29,28) & 11 &	     & 5 & 5 & (24,23) & 5  \\
     & 7 & 1 & (17,16) & 19 &	     & 1 & 5 & (24,23) & 25  \\
     & 7 & 1 & (10,9) & 33 & 	5913 & 1 & 1 & (41,40) & 73  \\
     & 7 & 1 & (6,5) & 57 & 	7373 & 1 & 1 & (100,1) & 73  \\
     & 7 & 1 & (1,0) & 627 & 	9065 & 49 & 5 & (19,18) & 1  \\ 
     & 1 & 1 & (6,5) & 399 & 	     & 49 & 5 & (1,0) & 37  \\ 
4453 & 1 & 1 & (31,30) & 73 &	     & 49 & 1 & (19,18) & 5 \\ 
&&&&&     				     & 49 & 1 & (3,2) & 37  \\
&&&&&     				     & 37 & 5 & (25,24) & 1  \\
&&&&&     				     & 37 & 5 & (1,0) & 49  \\
&&&&&     				     & 37 & 1 & (25,24) & 5  \\
&&&&&     				     & 37 & 1 & (3,2) & 49  \\
\end{array}
\]
\end{center}

There are four additional bad cases $n=191,5767,7081,8249$ that we
can take care of, using new results. The construction of $HM(4n)$ for $n=191$ appeared in \cite{DZ4}. 
For the remaining three cases we again apply 
Proposition \ref{fakt}. The acceptable choices
for the parameters are $y=h=1$, $(r,s)=(37,36)$ and
$w=79,97,113$, respectively. The main point is 
that we have shown that $NN(36)$ exist, and consequently 
$BS(37,36)$ exist (see the next section).

\section{The current status of the Yang conjecture} 
\label{40-Yang}

In his paper \cite{CHY} Yang said that "it is likely" that 
$NN(n)$ exist for all even integers $n>0$. This assertion has 
become known as ``Yang conjecture'', see e.g., \cite{HCD,DZ3}. 
Note that, in our notation which is different from
that of Yang, the set $NN(n)$ is empty for odd $n>1$.
In our recent paper \cite{DZ5} we have introduced
an equivalence relation for near-normal sequences, to which we 
refer as $NN$-equivalence. This leads to a canonical form for 
$NN$-equivalence which is too technical to be given here. 
By using this canonical form we were
able to enumerate the $NN$-equivalence classes in $NN(n)$ for even $n\le30$. Subsequently these exhaustive computations 
were extended to cover the cases of all even $n\le40$. 
For the cases $n=32$ and $n=34$ see our notes \cite{DZ6} and 
\cite{DZ7}, respectively. After finding out that there is only 
one $NN$-equivalence class for $n=36$, we lost any hope that 
Yang conjecture may be true in general. However, to our great 
surprise, it turned out that there is again a single 
$NN$-equivalence class in $NN(38)$, and the same holds 
true for $NN(40)$. The computations in the last two cases were 
carried out on SHARCNET's machines running at 3.0 GHz. 
The CPU time for the case $n=40$ was about 1300 days.

Here we give examples of $NN(n)$ for $n=32,34,36,38,40$ in our
encoded form:
\begin{eqnarray*}
n=32 &:& [07651732153537650,1262758654155332], \\
n=34 &:& [076417646512321462,16738541372344337], \\
n=36 &:& [0764841234846532153,165154775335162126], \\
n=38 &:& [07641237828515856281,1782612553714317675], \\
n=40 &:& [058214351717346462170,11868533752571536124].
\end{eqnarray*}
For the reader's convenience we also give an example of 
$BS(34,33)$:
$$ [07651732153537650,1262758654155332]. $$
The encoding scheme is explained in our papers \cite{DZ2,DZ3,DZ5}.

\section{Acknowledgments}

The author is grateful to NSERC for the continuing support of
his research. Part of this work was made possible by the facilities 
of the Shared Hierarchical Academic Research Computing Network 
(SHARCNET:www.sharcnet.ca).


\begin{thebibliography}{99}

\providecommand{\bysame}{\leavevmode\hbox to3em{\hrulefill}\thinspace}


\bibitem{HCD}
C.J. Colbourn and J.H. Dinitz, Editors, Handbook of Combinatorial Designs,
2nd edition, Chapman \& Hall, Boca Raton/London/New York, 2007.

\bibitem{DZ1}
D.\v{Z}. \DJo{}, 
Good matrices of orders 33, 35 and 127 exist,
J. Comb. Math. Comb. Comput. {\bf 14} (1993), 145--152.

\bibitem{DZ2}
\bysame, Base sequences, complementary ternary sequences, 
and orthogonal designs, 
J. Combinatorial Designs {\bf 4} (1996), 339--351.

\bibitem{DZ3}
\bysame, Aperiodic complementary quadruples of binary sequences,
JCMCC {\bf 27} (1998), 3--31. Correction: ibid {\bf 30} (1999), p. 254.

\bibitem{DZ4}
\bysame, Hadamard matrices of order $764$ exist, 
Combinatorica {\bf 28} (4) (2008), 487--489.

\bibitem{DZ5}
\bysame, Classification of near-normal sequences, 
Discrete Mathematics, Algorithms and Applications, 
{\bf 1}, No. 3 (2009), 389--399.
Available as a preprint on arXiv:0903.4390v2 [math.CO] 1 Sep 2009.

\bibitem{DZ6}
\bysame, Some new near-normal sequences, 
arXiv:0907.31290v1 [math.CO] 17 Jul 2009.

\bibitem{DZ7}
\bysame, A new Yang number and consequences,
Des. Codes Cryptogr. (to appear).

\bibitem{KS}
S. Kounias and K. Sotirakoglu, Construction of orthogonal sequences,
Proc. 14-th Greek Stat. Conf. 2001, 229--236 (in Greek).

\bibitem{SY}
J. Seberry and M. Yamada, Hadamard matrices, sequences and block designs,
in ``Contemporary Design Theory, A Collection of Surveys'',
J.H. Dinitz and D.R. Stinson, Eds., J. Wiley, New York, 1992.

\bibitem{CHY} C. H. Yang, On composition of four-symbol $\delta$-codes and
Hadamard matrices, 
Proc. Amer. Math. Soc. {\bf 107} (1989), 763--776.

\end{thebibliography}
\end{document}